\documentclass[12pt]{article}
\usepackage{amsfonts}
\usepackage{amsthm}
\usepackage{url}
\usepackage{mathtools}
\usepackage[all]{xy}
\usepackage{amssymb}
\usepackage{color}
\usepackage{mathrsfs}
\usepackage{tikz}
\usepackage{hyperref}
\usetikzlibrary{calc,intersections,through,backgrounds}

\usepackage{pdfsync}
\newtheorem{thm}{Theorem}[section]
\newtheorem{prop}[thm]{Proposition}
\newtheorem{lemma}[thm]{Lemma}
\newtheorem{coro}[thm]{Corollary}

\newcommand{\Nat}{{\mathbb N}}

\newcommand{\Real}{{\mathbb R}}

\makeatletter
\let\c@equation\c@thm
\makeatother
\numberwithin{equation}{section}

\title{A few remarks on the B\'{a}ez-Duarte Criterion}

\author{Alexandre Pyvovarov}

\date{\today}

\begin{document}

	\maketitle
	
	\begin{abstract}
		
	We study exponentially damped M\"obius approximants in $\mathscr H=L^2([1,\infty),dt/t^{-2})$. With \[ \gamma_n(t)=\left\lfloor\frac tn\right\rfloor -\frac{\lfloor t\rfloor}{n},\qquad f(u)(t)=\sum_{n\ge1}\mu(n)e^{-nu}\gamma_n(t),\] we compute the relevant scalar products, characterize the M\"obius coefficients as the unique coefficients giving pointwise convergence to the constant function, and prove $\langle1 \mid f(u)\rangle\to1$. Vasyunin's formula expresses $F(e^{-u})=\|f(u)\|_2^2$ as an arithmetic cotangent sum. In the final section we give a criterion that would help in the analysis of asymptotics of $F(x)$ as $x\uparrow 1$.
		
	\end{abstract}
	
	\tableofcontents
	
	\section{Setup}
	
	There are numerous statements equivalent to the Riemann Hypothesis. The core of this paper is the strong version of the Nyman-Beurling criterion. Let $\{ .\}$ be the fractional part and let $\mathcal{H}$ be the Hilbert space of the square integrable functions. Let $\mathcal{B}$ be the subspace of $\mathcal{H}$  generated by the family $\{x \mapsto \{ \frac{1}{ax} \} |  a \in \mathbb{R} , a \geq 1 \}$.
	
	The Nyman-Beurling dit que states that $\mathrm{\textbf{1}}_{]0,1]} \in  \overline{\mathcal{B}}$ is equivalent to the Riemann hypothesis, where $\overline{\mathcal{B}}$ is the closure of $\mathcal{B}$.
	
	The statement above can be replaced by a stronger condition. Let $\mathcal{B}^{nat}$ be the subspace of $\mathcal{H}$  generated by the family $\{x \mapsto \{ \frac{1}{ax} \} | a \in \mathbb{N}-\{0\} \}$. In 2002 B\'{a}ez-Duarte has proven in \cite{Baez2003} that $\mathrm{\textbf{1}}_{]0,1]} \in  \overline{\mathcal{B}^{nat}}$ is equivalent to the Riemann Hypothesis.
	
	\subsection{Notations. Definitions.}
	
	From now on let $k$ and $n$ strictly positive integers, let $\gamma$ be the Euler constant, and let $(n,k)$ be the gcd of $n$ and $k$.
	
	Let $\mathscr H=\{ f :[1,+\infty[\rightarrow \mathbb{C} | \int\limits_{1}^{+\infty} |f(t)|^{2} t^{-2} dt <  +\infty \}$ be the Hilbert space equipped with the scalar product $\langle f | g \rangle = \int\limits_{1}^{+\infty} f(t) \overline{g(t)} t^{-2} dt$. Let $\|.\|_{2}$ the associated norm.
	
	Let $r_{n}^{k}$ be the reminder of the euclidean division of $k$ by $n$ and let $q_{n}^{k}$ be the quotient. These integers are uniquely defined by  $k=q_{n}^{k}n + r_{n}^{k}$ and $0 \leq r_{n}^{k} \leq n-1$. Let $[.]$ be the floor function and let $\{.\}$ be the fractional part. Obviously $[\frac{k}{n}]=q_{n}^{k}$ and $\{\frac{k}{n}\} = \frac{r_{n}^{k}}{n}$. From now on these notations will be used interchangeably.
	
	Let us define two very useful elements of $\mathscr H$. If $\mathrm{Re}(z)>1/2$ then $t \mapsto t^{1-z} \in \mathscr H$ and $\forall n \geq 1$, let  $\gamma_{n} : [1,+\infty[\rightarrow \mathbb{R}$ be a function defined by $\gamma_{n}(t) = [\frac{t}{n}] - \frac{[t]}{n}$. Observe that $\forall n \geq 1$, $\gamma_{n} \in \mathscr H$. Indeed, $\forall t \geq 1 $, $|\gamma_{n}(t)| < 1$ and $t \mapsto t^{-2}$ is integrable on $[1,+\infty[$, therefore  $\forall n \geq 1$, $\gamma_{n} \in \mathscr H$.
	
	\subsection{A modified B\'{a}ez-Duarte criterion.}
	
	Let $B^{nat}$ be the subspace of $\mathscr H$ generated par $(\gamma_{n})_{n \geq 1}$. Let us observe that $1 \in \overline{B^{nat}}$, where $1$ is the constant function equal to  $1$, is equivalent to the Riemann Hypothesis according to the B\'{a}ez-Duarte criterion. A change of variable $x=1/t$ allows us to go from $\mathcal{H}$ to $\mathscr H$ .
	
	Moreover, there is a simple argument allowing to show that $1 \in \overline{B^{nat}}$ implies the Riemann Hypothesis. We begin with a representation $\zeta$ function by an integral:
	
	$$-\frac{\zeta(z)}{z}= \int\limits_{0}^{+\infty}x^{z-1}\{ \frac{1}{x} \} dx = \int\limits_{1}^{+\infty}t^{1-z}\{ t \} \frac{dt}{t^{2}}$$
	
	\noindent which holds for $1>\mathrm{Re}(z)>0$. By "iteration" the expression above we can derive $\langle t^{1-z} | \gamma_{n} \rangle= \left(\frac{1}{n^{z}}-\frac{1}{n} \right) \frac{\zeta(z)}{z}$. Let us observe that the integral on the left hand side, can be interpreted as a scalar product in $\mathscr H$ when $\mathrm{Re}(z)>1/2$.
	
	Now let's assume that $1 \in \overline{B^{nat}}$, therefore there is a sequence $(s_{n})$ of elements in $B^{nat}$ that converges to $1$ for the norm in $\mathscr H$.  Furthermore, let's assume that there is a complex number $z_{0}$ such that $\mathrm{Re}(z_{0})>1/2$ and $\zeta(z_{0})=0$. It follows: 
	
	$$\langle t^{1-z} | s_{n} \rangle=0$$
	
	When $n$ goes to infinity we get: $\frac{1}{z_{0}}=0$. Which is a contradiction. We must have $\mathrm{Re}(z)>1/2$, $\zeta(z)\neq 0$. The functional equation of the $\zeta$ function allows us to see that we also have $\mathrm{Re}(z)<1/2$, $\zeta(z)\neq 0$.
	
	So the statement $1 \in \overline{B^{nat}}$ implies the Riemann hypothesis. In what follows we will try to approximate $1$ by some linear combinations of elements of $B^{nat}$.
	
	\subsection{Mobius function and the prime number theorem}
	
	Let us recall a few elementary statements involving the Mobius function. The Mobius is defined by:
	
	\begin{equation}
		\mu(n) = \left\{\begin{array}{ll}
			1 & n=1 \\
			0 & p^2|n \\
			(-1)^{r} & n=p_1..p_r\\
		\end{array} \right.
	\end{equation}
	where, $p$ is a prime number and the $p_i$ are all distinct prime numbers.
	
	\begin{lemma}
		$\forall t \geq 1$ $\sum\limits_{n=1}^{+\infty} \mu(n)[\frac{t}{n}] =1$ or equivalently $\forall k \geq 1$ $\sum\limits_{n=1}^{k} \mu(n)q_{n}^{k} =1$.
	\end{lemma}
	
	\begin{proof}
		
		$\forall n \geq 2$, $n=\prod\limits_{i=1}^{r} p_{i}^{a_{i}}$ then,
		
		$$\sum_{d|n}\mu(d)= \mu(1) + \sum_{i_{1}<i_{2}}\mu(p_{i_{1}}p_{i_{2}})+...+\mu(p_{1}...p_{r})=\sum\limits_{j=0}^{r}(-1)^{j}=(1-1)^{r}=0$$
		
		\noindent and $\sum_{d|1}\mu(d)=1$. Let $f :[0,+\infty[\rightarrow \mathbb{C}$ such that $f(t)=0$ for $t<1$ and $g :[0,+\infty[\rightarrow \mathbb{C}$ defined by $g(t)= \sum\limits_{n=1}^{+\infty}f(\frac{t}{n})$. We have:
		
		$$\sum\limits_{n=1}^{+\infty}\mu(n)g(\frac{t}{n})=\sum_{n \leq x} \mu(n) \sum_{m \leq t/n}f(\frac{t}{mn})=\sum_{k \leq x} f(\frac{t}{k}) \sum_{mn=k }\mu(n)=f(t)$$
		
		Choosing $f$ such that $f(t)=0$ for $t<1$ and $f(t)=1$ for $t \geq 1$ we get the stated result.
	\end{proof}
	
	\medskip
	
	The prime number theorem is equivalent to the convergence of the series $\sum \frac{\mu(n)}{n}$ and $\sum_{n=1}^{+\infty} \frac{\mu(n)}{n} = 0$. See \cite{Ten95}.
	
	\medskip
	
	\begin{coro} \label{mobius}
		$\forall t \geq 1$ $\sum\limits_{n=1}^{+\infty} \mu(n)\gamma_{n}(t) =1$ and $\forall k \geq 1$ $\sum\limits_{n=1}^{+\infty} \frac{\mu(n)}{n} r_{n}^{k} =-1$.
	\end{coro}
	
	\medskip
	\noindent Let $\alpha(x)= \sum_{n \leq x} \frac{\mu(n)}{n}$, we have the following bound:
	\medskip
	
	\begin{lemma}\label{bigO}
		$\alpha(x)=O(e^{-c\sqrt{\log x}})$.
	\end{lemma}
	
	\begin{proof} The first effective Perron formula allows us to write, for $x \geq 1$, $T \geq 2$, $\kappa = 1/\log x$ :
		
		$$\alpha(x)=\frac{1}{2i\pi} \int\limits_{\kappa - iT}^{\kappa + iT}\zeta(1+w)^{-1}\frac{x^{w}}{w}dw + O \left( \sum\limits_{n=1}^{+\infty} \frac{n^{-\kappa-1}}{1+T|\log(x/n)|} \right) $$
		
		The contribution to the error term of the integers $n$ such that $|\log(x/n)| > 1$ is $\ll (\log x)/T$. The complementary contribution is $\ll \frac{1}{x} \sum\limits_{0 \leq h\leq x} \frac{1}{1+Th/x} \ll \frac{1}{x} + (\log x)/T$. Therefore:
		
		$$\alpha(x)=\frac{1}{2i\pi} \int\limits_{\kappa - iT}^{\kappa + iT}\zeta(1+w)^{-1}\frac{x^{w}}{w}dw + O \left( \frac{1}{x} + \frac{(\log x)}{T} \right) $$
		
		Moving the integration segment to the left until  $\mathrm{Re}(w)=-c/\log T$ (where $c$ is an absolute positive constant conveniently choses), then apllying the residue theorem and using the bound $1/\zeta(1-c/\log T +it) \ll \log T$ ($|t|\leq T$), we get:
		
		$$\alpha(x)= O \left( \frac{1}{x} + \frac{(\log x)}{T} + x^{-c/\log T}(\log T)^{2} \right) $$
		
		\noindent The choice $T=\exp(\sqrt{c\log x})$ gives us the desired result \end{proof}
	
	\subsection{Riemann zeta function and the scalar product}\label{prod}
	
	The goal of this section is to derive $\langle t^{1-z} | \gamma_{n} \rangle = \left(\frac{1}{n^{z}}-\frac{1}{n} \right) \frac{\zeta(z)}{z}$, where $\mathrm{Re}(z)>1/2$.
	
	\noindent We have $\forall n\geq 1$ and $\mathrm{Re}(z)>1$ :
	
	$$\langle t^{1-z} | \gamma_{n} \rangle $$
	
	$$= \int\limits_{1}^{+\infty}t^{1-z}\left( [\frac{t}{n}] - \frac{[t]}{n} \right) \frac{dt}{t^{2}}=\underbrace{\int\limits_{0}^{1}t^{1-z}\left( [\frac{t}{n}] - \frac{[t]}{n} \right) \frac{dt}{t^{2}}}_{=0}+\int\limits_{1}^{+\infty}t^{1-z}\left( [\frac{t}{n}] - \frac{[t]}{n} \right) \frac{dt}{t^{2}}$$

	$$=\int\limits_{0}^{+\infty}[\frac{t}{n}]\frac{dt}{t^{1+z}} - \int\limits_{0}^{+\infty}\frac{[t]}{n}\frac{dt}{t^{1+z}}= \left(\frac{1}{n^{z}}-\frac{1}{n}\right) \int\limits_{0}^{+\infty}[t]\frac{dt}{t^{1+z}}= $$
	
	$$=\left(\frac{1}{n^{z}}-\frac{1}{n}\right) \left( \underbrace{\int\limits_{0}^{1}[t]\frac{dt}{t^{1+z}}}_{=0} + \int\limits_{1}^{+\infty}[t]\frac{dt}{t^{1+z}}\right) = \left(\frac{1}{n^{z}}-\frac{1}{n}\right) \int\limits_{1}^{+\infty}[t]\frac{dt}{t^{1+z}}=$$
	
	$$= \left(\frac{1}{n^{z}}-\frac{1}{n}\right) \left( \sum\limits_{k=1}^{+\infty}\int\limits_{k}^{k+1} k \frac{dt}{t^{1+z}} \right) = \left(\frac{1}{n^{z}}-\frac{1}{n}\right) \frac{1}{z} \left( \sum\limits_{k=1}^{+\infty} k \left(\frac{1}{k^{z}}-\frac{1}{(k+1)^{z}} \right) \right)=$$
	
	$$=\left(\frac{1}{n^{z}}-\frac{1}{n}\right) \frac{1}{z} \left( \sum\limits_{k=1}^{+\infty}\left(\frac{k}{k^{z}}-\frac{k+1}{(k+1)^{z}} \right) + \frac{1}{(k+1)^{z}}\right)=$$
	
	$$=\left(\frac{1}{n^{z}}-\frac{1}{n}\right) \frac{1}{z} \left( \sum\limits_{k=1}^{+\infty}\left(\frac{k}{k^{z}}-\frac{k+1}{(k+1)^{z}} \right) + \sum\limits_{k=1}^{+\infty}\frac{1}{(k+1)^{z}}\right)$$
	
	$$=\left(\frac{1}{n^{z}}-\frac{1}{n}\right) \frac{1}{z} \left( 1 + \sum\limits_{k=2}^{+\infty}\frac{1}{k^{z}}\right)=\left(\frac{1}{n^{z}}-\frac{1}{n} \right) \frac{\zeta(z)}{z}$$
	
	\section{Interesting lemmas}
	
	According to \ref{mobius} the most natural way to approximate function $1$ would be to us the Mobius function.
	
	\subsection{In search of an approximation}
	
	We are going to for the sequences $(a_{n})_{n}$ such that the function $f_{N}^{a}=\sum\limits_{n=1}^{N} a_{n}\gamma_{n}$ converges point-wise to $1$.
	
	\medskip
	\medskip
	
	\begin{prop}
		A sequence of complex numbers $(a_{n})_{n}$ such that the function $f_{N}^{a}=\sum\limits_{n=1}^{N} a_{n}\gamma_{n}$ converges point-wise to  $1$ when $N\rightarrow +\infty$, are of the form : $(c,\mu(2),...,\mu(n),...)$, where $c$ is some constant.
	\end{prop}
	
	\begin{proof} Let $L(a,1)=\sum\limits_{n=1}^{+\infty}\frac{a_{n}}{n}$. This number exists and is well defined because  $f_{N}^{a}(1)=\sum\limits_{n=1}^{N} a_{n}\gamma_{n}(1)=-\sum\limits_{n=2}^{N}\frac{a_{n}}{n}$ goes to $1$ when $N\rightarrow +\infty$, by assumption. Let $T^{k}(a)= \sum_{n=1}^{k} a_{n}q_{n}^{k}$ and $X_{k}=\sum\limits_{n=1}^{+\infty}a_{n}\frac{r_{n}^{k}}{n}=k L(a,1)-T^{k}(a)$.
		
		Indeed, $X_{k}$ is constante. It is equivalent to say $\forall k\geq 1$, $X_{k}=X_{k+1}$
		\\
		$\Leftrightarrow$ $\forall k\geq 1$, $(k+1)L(a,1)-T^{k+1}(a)=k L(a,1)-T^{k}(a)$ 
		\\
		$\Leftrightarrow$ $\forall k\geq 1$, $L(a,1)=T^{k+1}(a)-T^{k}(a)$
		\\ 
		$\Leftrightarrow$ $\forall k\geq 1$, $kL(a,1)=\sum\limits_{n=1}^{k}T^{n+1}(a)-T^{n}(a)= T^{k+1}(a)-T^{1}(a)$
		\\ 
		$\Leftrightarrow$ $\forall k\geq 1$, $L(a,1)=\frac{1}{k}\sum\limits_{n=1}^{k+1}(a_{n}q_{n}^{k+1}-a_{1}):=F^{k}(a)$ 
		\\
		$\Leftrightarrow$ $\forall k\geq 1$, $F^{k}(a)=F^{1}(a)=2a_{1}+a_{2}-a_{1}$ and $F^{1}(a)=L(a,1)$
		\\
		$\Leftrightarrow$ $\forall k\geq 1$, $\sum\limits_{n=2}^{k+1} a_{n}q_{n}^{k+1}=ka_{2}$ and $F^{1}(a)=L(a,1)$
		
		If $a_{2}=0$, then we get by induction $\forall k\geq 1$, $\sum\limits_{n=2}^{k+1} a_{n}q_{n}^{k+1}=0$ we have $\forall n\geq 2$ $a_{n}=0$.
		
		If $a_{2}\neq 0$, let $b_{n}=a_{n}/(-a_{2})$ for $n \geq 2$, then $\sum\limits_{n=2}^{k+1} b_{n}q_{n}^{k+1}=-k$. Moreover, we have $\sum\limits_{n=2}^{k+1} \mu(n) q_{n}^{k+1}=-k$. Then $\forall k\geq 1$, $\sum\limits_{n=2}^{k+1} (b_{n}-\mu(n))q_{n}^{k+1}=0$, by induction we get $b_{n}=\mu(n)$ $\forall n\geq 2$. Comme $-\sum\limits_{n=2}^{+\infty}\frac{a_{n}}{n}=1$, $a_{2}=-1$. The coefficients $a_{n}$ are of the desired form.
		
		Reciprocally, the sequence $(c,\mu(2),...,\mu(n),...)$ satisfies all the desired conditions \end{proof}.
	
	\medskip
	
	We could take $a_{n}=\mu(n)$ but B\'{a}ez-Duarte has shown that the natural approximation $\sum \mu(n) \gamma_{n}$ diverges in $\mathscr H$. A more detailed account can be found in \cite{BAEZDUARTE19991} and \cite{BD2002}. Therefore, one has to look for the coefficients $a_{n}$ that depend on a parameter, which in turn may depend on $n$.

	\subsection{Exponential dampening}
	
	The idea is to pick a sequence in $\overline{B^{nat}}$ that converges to $1$. Let us consider a family of sequences that depend upon a parameter $u \in ]0,+\infty]$, and are defined by:
	
	$$ f(u) : t \mapsto \sum_{n=1}^{+\infty} \mu(n)e^{-nu} \gamma_{n}(t)$$
	
	Let us observe that $\forall u \in ]0,+\infty]$ and $\forall N \geq 1$:
	
	$$\left\| \sum_{n=1}^{N} \mu(n)e^{-nu} \gamma_{n}-\sum_{n=1}^{+\infty} \mu(n)e^{-nu} \gamma_{n} \right\| _{2}^{2}$$
	$$=\int\limits_{1}^{+\infty}\left| \sum_{n=N+1}^{+\infty} \mu(n)e^{-nu} \gamma_{n}(t)\right|^{2} \frac{dt}{t^{2}} \leq \left( \sum_{n=N+1}^{+\infty} e^{-nu}\right)^{2} $$
	
	This quantity goes to $0$ when $N\rightarrow +\infty$. Then $ \forall u \in ]0,+\infty]$ $f(u) \in \overline{B^{nat}}$. All of this boils down to check the continuity of the norm of $f(u)$ in $0$.
	
	For all integers $k\geq 1$ and $\forall u \in ]0,+\infty]$ let $h_{k}(u)=\sum\limits_{n=1}^{+\infty}\mu(n)\frac{r_{n}^{k}}{n}e^{-nu}$.
	
	\begin{lemma}  
		$\forall u \in ]0,+\infty[$, $\left\|f(u)\right\| _{2}^{2}=\sum\limits_{k=1}^{+\infty} \frac{1}{k(k+1)}|h_{k}(u)|^{2}$.
	\end{lemma}
	
	\begin{proof} We fix an $u>0$. By definition $\left\|f(u)\right\| _{2}^{2}=\int\limits_{1}^{+\infty}\left| \sum_{n=1}^{+\infty} \mu(n)e^{-nu} \gamma_{n}(t)\right|^{2} \frac{dt}{t^{2}}$ 
		$=\sum\limits_{k=1}^{+\infty}\int\limits_{k}^{k+1}\left| \sum_{n=1}^{+\infty} \mu(n)e^{-nu} \gamma_{n}(t)\right|^{2} \frac{dt}{t^{2}}= \sum\limits_{k=1}^{+\infty} \frac{1}{k(k+1)}|-h_{k}(u)|^{2}$ \end{proof}.
	
	\begin{lemma} 
		$\forall u \in ]0,+\infty[$,  $\langle 1 | f(u) \rangle = - \sum\limits_{k=1}^{+\infty} \frac{1}{k(k+1)}h_{k}(u)$.
	\end{lemma}
	
	\begin{proof} Similarly as before \end{proof}.
	
	\begin{lemma}  
		$\forall u \in ]0,+\infty[$,  $\langle 1 | f(u) \rangle = - \sum\limits_{n=1}^{+\infty} \frac{\mu(n)}{n}\log(n)e^{-nu}$.
	\end{lemma}
	
	\begin{proof} Let's compute the following expressions:
		
		\noindent \textbf{1)} We have $\sum_{k=1}^{n-1}t^{k}=\frac{t-t^{n}}{1-t}$ and by differentiation: $\sum_{k=1}^{n-1}kt^{k-1}=\frac{1-nt^{n-1}}{1-t}+\frac{t-t^{n}}{(1-t)^2}$
		
		\noindent \textbf{2)} First: $\sum_{k=1}^{+\infty}\frac{r_{n}^{k}}{k(k+1)}=\sum_{k=1}^{n-1}\frac{k}{k(k+1)}+\sum_{p=1}^{+\infty}\sum_{k=pn+1}^{pn+n}\frac{r_{n}^{k}}{k(k+1)}= \sum_{k=1}^{n-1}\frac{k}{k(k+1)}+\sum_{p=1}^{+\infty}\sum_{k=1}^{n-1}\frac{k}{(pn+k)(pn+k+1)}= \sum_{k=1}^{n-1}\sum_{p=0}^{+\infty}\frac{k}{(pn+k)(pn+k+1)}$
		
		\noindent Then:
		$$\sum_{p=0}^{N}\frac{1}{(pn+k)} =  \sum_{p=0}^{N} \int_{0}^{1} t^{pn+k-1}dt = \int_{0}^{1} t^{k-1}\frac{1-t^{n(N+1)}}{1-t^{n}}dt$$
		
		\noindent Therefore:
		
		$$\sum_{p=0}^{N}\frac{1}{(pn+k)(pn+k+1)}=\sum_{p=0}^{N}\frac{1}{(pn+k)}-\frac{1}{(pn+k+1)} $$ 
		$$= \int_{0}^{1} t^{k-1}\frac{1-t^{n(N+1)}}{1-t^{n}}(1-t)dt$$
		
		\noindent $\forall N \geq 0$ et $\forall t \in ]0,1[ $, $\left|t^{k-1} \frac{1-t^{n(N+1)}}{1-t^{n}}(1-t) \right| \leq t^{k-1}\frac{1-t}{1-t^{n}} $. The bound is integrable on $]0,1[$. It follows from the dominated convergence theorem that:
		
		$$\sum_{p=0}^{+\infty}\frac{1}{(pn+k)(pn+k+1)}=\int_{0}^{1} t^{k-1}\frac{1-t}{1-t^{n}}dt$$
		
		\noindent Using the expressions from 1):
		
		$$\sum_{k=1}^{n-1}\sum_{p=0}^{+\infty}\frac{k}{(pn+k)(pn+k+1)}$$
		$$= \int_{0}^{1} \frac{1-nt^{n-1}}{1-t^{n}} + \frac{t-1+1-t^{n}}{(1-t^{n})(1-t)} dt = \int_{0}^{1} \frac{1-nt^{n-1}}{1-t^{n}} - \frac{1}{(1-t^{n})} + \frac{1}{(1-t)}dt$$
		
		Furthermore, let us observe that for $\varepsilon \in ]0,1[$, we have:
		
		\begin{center}
			$\int_{0}^{\varepsilon} \frac{nt^{n-1}}{1-t^{n}}dt=-\log(1-\varepsilon^{n})$ and $\int_{0}^{\varepsilon} \frac{1}{1-t}dt=-\log(1-\varepsilon)$
		\end{center}
		\noindent we deduce that:
		$$\int_{0}^{1} \frac{-nt^{n-1}}{1-t^{n}}  + \frac{1}{(1-t)}dt$$
		$$ = \lim\limits_{\varepsilon \rightarrow 1}\int_{0}^{\varepsilon} \frac{-nt^{n-1}}{1-t^{n}}  + \frac{1}{(1-t)}dt=\lim\limits_{\varepsilon \rightarrow 1} \log \frac{(1-\varepsilon^{n})}{(1-\varepsilon)}= \lim\limits_{\varepsilon \rightarrow 1} \log(\sum_{k=0}^{n-1}\varepsilon^{k})= \log(n)$$

		\textbf{3)} Grouping together the results above, we get: $\sum_{k=1}^{+\infty}\frac{r_{n}^{k}}{k(k+1)}=\log(n)$. $\forall k\geq 1$, $\forall n\geq 1$, $\forall u \in ]0,+\infty[$, $\left|\frac{\mu(n)r_{n}^{k}e^{-nu}}{nk(k+1)} \right| < \frac{e^{-nu}}{k(k+1)}$ then the double sum satisfies all the conditions fo the Fubini theorem. It follows that we are allowed to exchange the summation order and get the desired result. \end{proof}
	
	\begin{coro}
		For all integers $n\geq 1$, we have $\sum_{k=1}^{+\infty}\frac{r_{n}^{k}}{k(k+1)}=\log(n)$.
	\end{coro}
	
	\subsection{A few statements about the continuity}
	
	\begin{prop}
		$\forall k\geq 1$, $h_{k}(u) : u \mapsto \sum_{n=1}^{+\infty} \frac{\mu(n)}{n}r_{n}^{k}e^{-nu}$ converges uniformly on $[0,+\infty[$, therefore it deines a continuous function on $[0,+\infty[$.
	\end{prop}
	
	\begin{proof} Let $\forall k\geq 1$ $a_{n}(k)=\frac{\mu(n)}{n}r_{n}^{k}$, $b_{n}(u)=e^{-nu}$ and $A_{n,N}(k)=\sum\limits_{j=1}^{n}a_{j+N}(k)$. Let us fiw an integer $k$. Using the Abel transform we get:
		
		$$\sum\limits_{n=N+1}^{N+p} a_{n}(k) b_{n}(u) $$
		$$= \sum\limits_{n=1}^{p} a_{n+N}(k) b_{n+N}(u)=A_{p,N}(k)b_{p+N}(u)+\sum\limits_{n=1}^{p-1} A_{n,N}(k)(b_{n+N}(u)-b_{n+N+1}(u))$$
		
		The series $\sum a_{n}(k)$ converges according to the prime number theorem, therefore the sequence of the partial sums is a Cauchy sequence:
		
		\begin{center}
			$\forall \varepsilon$, $\exists N_{0} \in \mathbb{N}$ such that $\forall N \geq N_{0}$, $\forall p \geq 1$ $ \Longrightarrow$ $ \left| \sum_{n=N+1}^{N+p} a_{n}(k) \right|  < \varepsilon$
		\end{center}
		
		\noindent Again by Abel summation:
		\begin{center}
			$\forall \varepsilon$, $\exists N_{0} \in \mathbb{N}$ such that $\forall N \geq N_{0}$, $\forall p \geq 2$,$\forall u \in [0,+\infty[$  $ \Longrightarrow$ 
			
			$ \left| \sum_{n=N+1}^{N+p} \frac{\mu(n)}{n}r_{n}^{k}e^{-nu} \right| \leq \left\lbrace b_{p+N}(u)+\sum\limits_{n=1}^{p-1}(b_{n+N}(u)-b_{n+N+1}(u)) \right\rbrace \varepsilon = b_{N}(u)\varepsilon \leq \varepsilon$
		\end{center}
		
		The series $h_{k}$ satisfies the Cauchy criterion on $[0,+\infty[$, therefore it converges uniformly on $[0,+\infty[$. The continuity follows. \end{proof}
	
	\medskip
	
	Observe that $\forall k\geq 1$, $h_{k}(0)=-1$.
	
	\medskip
	
	\begin{prop}
		The function, $ u \langle 1 | f(u) \rangle = - \sum\limits_{n=1}^{+\infty} \frac{\mu(n)}{n}\log(n)e^{-nu}$ converges uniformly on $[0,+\infty[$, therefore is continuous on $[0,+\infty[$.
	\end{prop}
	
	\begin{proof} 
		
		\textbf{1)} let $x>1$.
		
		$$\zeta(x)=\sum_{n=1}^{+\infty} \frac{1}{n^{x}}=\int_{1}^{+\infty}\frac{dt}{t^{x}}+\sum_{n=1}^{+\infty} \int_{n}^{n+1}(n^{-x}-t^{-x})dt$$
		$$=\frac{-1}{x-1}+\sum_{n=1}^{+\infty} \int_{n}^{n+1}(n^{-x}-t^{-x})dt$$
		
		As we have $0<(n^{-x}-t^{-x})=\int_{n}^{t}xu^{-x-1}du < \frac{x}{n^{2}}$ for $n<x<n+1$, we get $0<\int(n^{-x}-t^{-x})dt < \frac{x}{n^{2}}$. Then:
		
		$$\zeta(x) \sim_{x \rightarrow 1} \frac{1}{1-x}$$
		
		\noindent Similarly:
		
		$$-\zeta'(x)=\sum_{n=1}^{+\infty} \frac{\log(n)}{n^{x}}=\int_{1}^{+\infty}\log(t)t^{-x}dt+\sum_{n=1}^{+\infty} \int_{n}^{n+1}(\log(n)n^{-x}-\log(t)t^{-x})dt$$
		
		The change in variable $u=\log(t)$ allows to compute easily: 
		
		$$\int_{1}^{+\infty}\log(t)t^{-x}dt=\frac{1}{(x-1)^2}.$$

		As $0<\log(n)n^{-x}-\log(t)t^{-x}=\int_{n}^{t}u^{-x-1}(x\log(u)-1)du < \frac{x}{n^{2}}(\log(n)+1)$, is valid for $3\leq n < x < n+1$. It follows:
		
		$$\zeta'(x) \sim_{x \rightarrow 1} \frac{-1}{(x-1)^2}$$
		
		We also have $\forall x>1$:
		$$-\sum_{n=1}^{+\infty}\frac{\mu(n)}{n^{x}}\log(n)= \frac{d}{dx}\sum_{n=1}^{+\infty}\frac{\mu(n)}{n^{x}}= \frac{d}{dx} \frac{1}{\zeta(x)}=-\frac{\zeta'(x)}{\zeta(x)^{2}} \sim_{x \rightarrow 1} 1$$

		\textbf{2)} Let $N$ and $P$ two strictly positive integers. Summing by parts:
		
		$\sum_{n=N+1}^{N+P}\frac{\mu(n)}{n}\log(n)=-\sum_{n=N+1}^{N+P-1} \alpha(n)\log(1+\frac{1}{n})+  \alpha(N)\log(N)+f(N+P)\log(N+P)$. 
		
		As $\left|\alpha(n)\log(1+\frac{1}{n}) \right| = O \left( \frac{e^{-c\sqrt{\log n}}}{n} \right)$ according to Lemma \ref{bigO} and
		
		$\int_{1}^{+\infty}\frac{e^{-c\sqrt{\log t}}}{t} dt =\int_{0}^{+\infty}e^{-c\sqrt{y}}dy < +\infty$. 
		
		Then the series $\sum_{n \geq 1} \alpha(n)\log(1+\frac{1}{n})$ is absolutely convergent because the series $\sum_{n \geq 1} \frac{e^{-c\sqrt{\log n}}}{n}$ converges by comparison of series and integral. In particular $\sum_{n \geq 1} \alpha(n)\log(1+\frac{1}{n})$ satisfies the Cauchy criterion:
		
		Let $\varepsilon > 0$,
		\begin{center}
			$\exists N_{0} \in \mathbb{N}$ such that $\forall N \geq N_{0}$, $\forall P \geq 2$ $ \Longrightarrow$ $ \left| \sum_{n=N+1}^{N+P-1} \alpha(n)\log(1+\frac{1}{n}) \right|  < \varepsilon$
		\end{center}
		
		According to Lemma \ref{bigO}, $\alpha(N)\log(N) \rightarrow 0$ when $N\rightarrow +\infty$. Then $\exists N_{1} \in \mathbb{N}$ such that $\forall N \geq N_{1}$ $ \Longrightarrow$ $ \left| \alpha(N)\log(N) \right|  < \varepsilon$
		
		Therefore $\forall N \geq \sup(N_{0},N_{1})$, $\forall P \geq 2$ $ \Longrightarrow$ $ \left| \sum_{n=N+1}^{N+P}\frac{\mu(n)}{n}\log(n) \right|  < 3\varepsilon$
		
		Then according to the Cauchy criterion the series $\sum_{n \geq 1}\frac{\mu(n)}{n}\log(n)$ converges.
		
		\textbf{3)} Similarly as in previous proposition we could use the Abel transform to show the uniform convergence of the function $x \mapsto -\sum_{n=1}^{+\infty}\frac{\mu(n)}{n^{x}}\log(n)$ on $[1,+\infty[$, beacuse $\sum_{n \geq 1}\frac{\mu(n)}{n}\log(n)$ converges. 
		
		And, by continuity $-\sum_{n=1}^{+\infty}\frac{\mu(n)}{n^{x}}\log(n)=1$. Then same method allows us to show the uniform convergence of $u \mapsto \langle 1|f(u) \rangle$ on $[0,+\infty[$, in particular we have $\lim\limits_{u\rightarrow 0} \langle 1|f(u) \rangle = 1$. \end{proof}
	
	\medskip
	\noindent If $\lim\limits_{u\rightarrow 0} \left\| f(u) \right\| _{2}^{2} = 1$, then:
	
	$$\left\| f(2^{-n}) -1\right\| _{2}^{2}=\left\| f(2^{-n}) \right\| _{2}^{2}-2\langle 1|f(2^{-n}) \rangle+1$$
	
	\noindent goes to $0$ when $n\rightarrow +\infty$. in this case $(f(2^{-n}))_{n \in \mathbb{N}} \in \overline{B^{nat}}^{\mathbb{N}}$ is a sequence that converges to $1$ in $\mathscr H$, as $\overline{B^{nat}}$ is obviously closed, $1 \in \overline{B^{nat}}$. However, $\lim\limits_{u\rightarrow 0} \left\| f(u) \right\| _{2}^{2} = 1$ might not be possible.
	
	\subsection{Power series}
	
	\medskip
	When exchanging the summation order in $\left\| f(u) \right\| _{2}^{2}$, we get the following
	\medskip
	
	\begin{prop} $\forall x \in ]0,1[$ :
		
		$$F(x):=\left\| f(-\log(x)) \right\| _{2}^{2}$$
		
		$$ = \left( \sum\limits_{n=2}^{+\infty}\mu(n)x^{n} \right)\times\left( \sum\limits_{n=2}^{+\infty}\frac{\mu(n)}{n} \log(n) x^{n}\right)$$
		
		$$-\left( \sum\limits_{n=2}^{+\infty}\frac{\mu(n)}{n} x^{n}\right)\times\left( \sum\limits_{n=2}^{+\infty}\frac{\mu(n)}{n} \log(n) x^{n}\right) $$
		
		$$-\pi \left( \sum\limits_{n=2}^{+\infty}\frac{\mu(n)}{n} x^{n}\right)\times \left( \sum\limits_{n=2}^{+\infty}\frac{\mu(n)}{n}\left( \sum\limits_{k=1}^{n-1}\left(\frac{1}{2}-\frac{k}{n} \right) \cot(\frac{\pi k}{n}) \right) x^{n}\right) $$
		
		$$+\pi \sum\limits_{n=2}^{+\infty}\sum\limits_{m=2}^{+\infty}\frac{\mu(n)}{n}\frac{\mu(m)}{m} A(n,m)x^{n+m},$$
		
		\noindent where
		
		$$A(n,m)=(n,m)\sum\limits_{k=1}^{\frac{n}{(n,m)}-1} \left(\frac{1}{2}-\{\frac{km}{n}\} \right) \cot(\frac{\pi k (n,m)}{n})$$
	\end{prop}
	
	\begin{proof} $\forall m \geq 2$, $\forall n \geq 2$ let $\omega=(n,m)$, $m=\omega m_{0}$ et $n=\omega n_{0}$. According to the expressions proven in \cite{Vasyunin1996}, we get the following formulas for the scalar products
		
		$$nm\langle  \gamma_{n} | \gamma_{m}  \rangle $$
		$$= \frac{m-1}{2}\log(n) + \frac{n-1}{2} \log(m)- \frac{\pi}{2}\sum\limits_{k=1}^{n-1} \left(\frac{1}{2}-\frac{k}{n} \right) \cot(\frac{\pi k}{n})$$
		
		$$ - \frac{\pi}{2}\sum\limits_{k=1}^{m-1} \left(\frac{1}{2}-\frac{k}{m} \right) \cot(\frac{\pi k}{m})$$
		
		$$+\frac{\pi \omega}{2}\sum\limits_{k=1}^{n_{0}-1} \left(\frac{1}{2}-\{\frac{km_{0}}{n_{0}}\} \right) \cot(\frac{\pi k}{n_{0}})+\frac{\pi \omega}{2}\sum\limits_{k=1}^{m_{0}-1} \left(\frac{1}{2}-\{\frac{kn_{0}}{m_{0}}\} \right) \cot(\frac{\pi k}{m_{0}})$$
		
		Observe that 
		$A(n,m)=\sum\limits_{k=1}^{n_{0}-1} \left(\frac{1}{2}-\{\frac{km_{0}}{n_{0}}\} \right) \cot(\frac{\pi k}{n_{0}})$ 
		
		\noindent and 
		
		$A'(n,m)=\sum\limits_{k=1}^{m_{0}-1} \left(\frac{1}{2}-\{\frac{kn_{0}}{m_{0}}\} \right) \cot(\frac{\pi k}{m_{0}})=A(m,n)$. 
		
		The formula for $F(x)$ can be easily inferred from the expansion of the scalar product: 
		
		$$F(x) = \left\langle \sum_{n=2}^{+\infty} \mu(n)x^{n} \gamma_{n} | \sum_{m=2}^{+\infty} \mu(m) x^{m} \gamma_{m} \right\rangle $$
		
	\end{proof}

	\medskip
	Furthermore, Vasyunin has shown the following asymptotic formula:
	
	$$\frac{\pi}{n}\sum\limits_{k=1}^{n-1}\left(\frac{1}{2}-\frac{k}{n} \right) \cot(\frac{\pi k}{n})=\log(n)+\gamma -\log(2\pi)+O(1/n)$$
	
	Put
	\[
	\begin{aligned}
		M(x)&:=\sum_{n\geq1}\mu(n)x^n,
		&m(x)&:=\sum_{n\geq1}\frac{\mu(n)}n x^n,\\
		L(x)&:=\sum_{n\geq1}\mu(n)\log n\,x^n,
		&l(x)&:=\sum_{n\geq1}\frac{\mu(n)\log n}{n}x^n.
	\end{aligned}
	\]
	Heuristically, replacing the cotangent sums in $F(x)$ by asymptotic formulas, we get the following function:
	
	\[( M(x)-x ) l(x) - (m(x)-x)l(x)\]
	\[ - (m(x)-x)\sum_{n \geq 2}\mu(n)x^n( \log n +\gamma -\log(2\pi)) \]
	
	\[+ \sum_{n \geq 2}\sum_{m \geq 2}\mu(n)x^n \frac{\mu(m)}{m} x^m ( \log n - \log m  +\gamma -\log(2\pi) )\]
	
	\[= ( M(x)-x ) l(x) - (m(x)-x)l(x) - (m(x)-x) (L(x) + \gamma -\log(2\pi)(M(x)-1))\]
	
	\[+ (m(x)-x)(L(x) + \gamma -\log(2\pi)(M(x)-x)) - (M(x)-x)l(x)\]
	
	\[ = (x-m(x))l(x)\]
	
	We see that the linear divergent terms cancel out. Heuristically $F = (id - m)l + err$, where $err$ is some unknown function. Then $F$ should behave as $err$. In the last section of this paper we will give some evidence that $F$ can not be bounded.
	
	\begin{lemma}
		Assume that $L=\lim\limits_{x\to 1}F(x)$ exists, then $L\geq 1$.
	\end{lemma}
	
	\begin{proof}
	$$0<||f(u)-1||^2 = F(x) - 2 +1$$, taking the limit we get
	\end{proof}
	
	\subsection{Mobius function and polynomials}
	
	For $x\in[0,1]$, put
	\[
	R_k(x):=\sum_{n\mid k}\mu(n)x^n,
	\qquad
	Q_k(x):=\frac1k\sum_{n=1}^{k}\mu(n)
	\left\lfloor\frac kn\right\rfloor x^n,
	\]
	and
	\[
	g(x):=\sum_{n=1}^{\infty}\frac{\mu(n)}n x^n,
	\qquad h_k(x)=h_k(u):=k\bigl(g(x)-Q_k(x)\bigr).
	\]
	The prime number theorem implies that $g$ is continuous on $[0,1]$.
	The polynomials $Q_k$ have some useful properties.
	
	\begin{prop}
		Set $R_1(x) = Q_1(x)=x$. We have $(k+1)Q_{k+1} - kQ_k = R_{k+1}$ and $Q_k = \frac{1}{k}\sum_{i=1}^{k} R_i$. Moreover, $Q_k$ converges uniformly on every compact of $[0,1[$ to the function $g$.
	\end{prop}
	
	\begin{proof}
		The relation $(k+1)Q_{k+1} - kQ_k = R_{k+1}$ follows from the observation if $n|k+1$, then $\left[ \frac{k+1}{n}\right] =  \left[ \frac{k}{n}\right] +1$, otherwise $\left[ \frac{k+1}{n}\right] =  \left[ \frac{k}{n}\right]$. Moreover, by induction we get that $Q_k$ is the average of the $R_i$. 
		
		Let $0<c<1$. For the convergence it is enough to observe that $|\mu(n)[k/n](x^n)/k|<c^n$, for all $x \in [0,c]$, and 
		$$\lim\limits_{k \to + \infty} Q_k(x) = \sum_{n=1}^{+\infty} \lim\limits_{k \to + \infty} \mu(n)[k/n](x^n)/k = g(x).$$
	\end{proof}
	
	The function  $g$ and its derivative were already studied in the literature, for example \cite{FRO}. According to the paper above, the function $g$ is concave on  $[0 1-\delta]$, and then it starts to oscillate infinitely often, where  $\delta = 10^{-18}$. So the  $Qk$'s can not be concave on the whole $[0,1]$, in general. Furthermore, one can easily check that $Q_{10}(0.9) \approx 1.665 \times 10^{-3}$. An interesting question would be to investigate how the polynomials $Q_k$ change the convexity on $[0,1]$. Let $p$ be a prime number then $R_p(x) = x - x^p$ is concave on $[0,1]$, however for $R_{p_1...p_r}$ this is no longer the case, where $r\geq 2$ and $p_i$ are distinct primes.
	
	\begin{lemma}
		We have
		\[ R_{p_1...p_r}(x) = \sum_{S \subseteq \left\lbrace p_2,...,p_r \right\rbrace } (-1)^{|S|} R_{p_1}(x^{\prod_{p \in S} p}), \]
		\noindent where $r\geq 2$ and $p_i$ are distinct primes.
	\end{lemma}
	
	It was quite unexpected to notice that the averages of the polynomials $R_i$ converge to $g$. How good is this approximation? The goal of this section is to give an unconditional bound for $\|g-Q_k\|= \sup_{x \in [0,1]} |g(x)-Q_k(x)|$, and to give a second proof of a result proven above.
	
	Let $M(y)  = \sum_{n\leq y} \mu(n)$. The classical
	Korobov--Vinogradov zero-free region
	\cite{Vinogradov1958},\cite{Korobov1958} yields constants $C>0$ and $c_0>0$
	such that, for all sufficiently large $y$,
	
	$$|M(y)| \leq C y \exp(-c_0 \frac{(\log y)^{3/5}}{(\log \log y)^{1/5}}).$$
	
	For convenience, write
	
	\[
	\Phi(y):=\frac{(\log y)^{3/5}}{(\log\log y)^{1/5}},
	\qquad E_c(y):=\exp\bigl(-c\Phi(y)\bigr).
	\]
	
	\begin{lemma}\label{sup}
		Let $S_y(x) : = \sum_{n \leq y} \mu(n) x^n$. There exists a constant $c_1>0$ such that:	
		$$ \sup_{x\in [0,1]} \left| S_y(x) \right| \ll y E_{c_1}(y)$$
	\end{lemma}
	
	\begin{proof}
		The Abel summation gives
		
		$$ S_y(x) = M(y)x^y + (1-x)\sum_{n<y} M(n) x^n$$
		
		The first term satisfies
		
		$$ | M(y)x^y | \ll y E_{c_0}(y)$$
		
		Let
		\[
		\Sigma_1:=(1-x)\sum_{n\leq\sqrt y}M(n)x^n,
		\qquad
		\Sigma_2:=(1-x)\sum_{\sqrt y<n<y}M(n)x^n.
		\]
		Using $|M(n)|\leq n$, we get
		
		\[ \Sigma_1 \leq (1-x) \sum_{n\leq \sqrt (y)} n x^n \leq \sqrt(y) \]
		
		For $\sqrt(y) < n < y$, we have:
		
		\[ M(n)\ll n E_{c_0}(n) < n E_{c_0}(\sqrt(y)) \]
		
		Then, 
		
		\[ \Sigma_2 \ll E_{c_0}(\sqrt(y)) (1-x)\sum_{n<y}n x^n < y E_{c_0}(\sqrt(y)) (1-x)\sum_{n<y} x^n \leq y E_{c_0}(\sqrt(y)) \]
		
		Now $\Phi(\sqrt(y)) \sim 2^{-3/5} \Phi(y)$, so there is a constant $c_1>0$ such that $E_{c_0}(\sqrt(y)) \ll E_{c_1}(y)$, and also $\sqrt(y) \ll yE_{c_1}(y)$. Therefore
		
		\[ |S_y(x)| \ll y E_{c_1}(y) \]
		
		\noindent uniformly in x.
	\end{proof}
	
	For $0<t\leq 1$, define $A_x(t):=\sum_{n\leq 1/t} \mu(n) x^n$.
	
	\begin{lemma}
		For every $x\in[0,1]$,
		\[
		\int_0^1A_x(t)\,dt=g(x),
		\qquad
		\sup_{x\in[0,1]}|A_x(t)|\ll\frac1tE_{c_1}(1/t).
		\]
	\end{lemma} 
	
	\begin{proof}
		Let $1_{[1,1/t]}$ be the characteristic function of the interval $[1,1/t]$, then
		
		\begin{align*}
			\int_0^1A_x(t)\,dt
			&=\int_0^1\sum_{n\geq1}\mu(n)x^n
			1_{[1,1/t]}(n)\,dt\\
			&=\sum_{n\geq1}\mu(n)x^n
			\int_0^1 1_{[1,1/t]}(n)\,dt\\
			&=\sum_{n\geq1}\mu(n)x^n\int_0^{1/n}dt=g(x).
		\end{align*}
		
		The other statement follows from the previous lemma.
	\end{proof}
	
	\begin{lemma}
		We have $Q_k(x) = \frac{1}{k} \sum_{m=1}^{k} A_x\left( \frac{m}{k} \right)$. So $Q_k$ is the Riemann sum approximating to the integral $\int_{0}^{1} A_x(t)dt$.
	\end{lemma}
	
	\begin{proof}
		\begin{align*}
			\frac1k\sum_{m=1}^k A_x\left(\frac mk\right)
			&=\frac1k\sum_{n\geq1}\mu(n)x^n
			\sum_{m=1}^k1_{[1,k/m]}(n)\\
			&=\frac1k\sum_{n\geq1}\mu(n)x^n
			\left|\{m:m\leq k, m\leq k/n\}\right|\\
			&=\frac1k\sum_{n\geq1}\mu(n)x^n
			\left\lfloor\frac kn\right\rfloor=Q_k(x).
		\end{align*}
		
	\end{proof}
	
	\begin{lemma}\label {Eint}
		For every $c>0$, there exists $c'\in]0, c[$ such that
		
		\[ \int_{N}^{+\infty}\frac{E_{c}(u)}{u} du \ll E_{c'}(N) \]
	\end{lemma}
	
	\begin{proof}
		Set $v=\log u$, and let $\psi(v) = \frac{v^{3/5}}{(\log v)^{1/5}}$ then the integral becomes
		
		\[ \int_{\log N}^{+\infty} \exp( -c \psi(v)) dv  = \int_{\log N}^{+\infty} -\frac{1}{c} \frac{1}{\psi'(v)}\frac{d}{dv}\exp( -c \psi(v)) dv \]
		
		Integration by parts gives us
		\[ \int_{\log N}^{+\infty} e^{ -c \psi(v)} dv = \frac{e^{-c\psi(\log N)}}{\psi'(\log N)} + \frac{1}{c} \int_{\log N}^{+\infty} e^{ -c \psi(v)} \frac{\psi''(v)}{(\psi'(v))^2} dv\]
		
		A direct calculation gives
		\[
		\frac{\psi''(v)}{(\psi'(v))^2}
		=v^{-3/5}(\log v)^{1/5}
		\frac{-6-(\log v)^{-1}+6(\log v)^{-2}}
		{(3-(\log v)^{-1})^2}.
		\]
		In particular, this function is bounded.  The integration-by-parts
		formula therefore gives
		
		\[ \int_{\log N}^{+\infty} e^{ -c \psi(v)} dv \ll  \frac{e^{-c\psi(\log N)}}{\psi'(\log N)} \ll (\log N)^{2/5}(\log \log N)^{1/5} e^{-c\psi(\log N)} \ll E_{c'}(N) \]	
	\end{proof}
	
	\begin{lemma}\label{Esum}
		For every $c>0$, there exists $c'\in]0,c[$ such that, uniformly for integers $k\geq N \geq 5$,
		\[ \sum_{m\leq k/N}\frac{1}{m} E_{c}\left( \frac{k}{m}\right) \ll_{c} E_{c'}(N) \]
		The implicit constant is independent of $k$ and $N$.
	\end{lemma}
	
	\begin{proof}
		For sufficiently large $u$, the function $\Phi$ is increasing. Indeed, writing $v=\log u$, let $\Phi(u)=\psi(v) = \frac{v^{3/5}}{(\log v)^{1/5}}$, then 
		
		\[ \psi'(v) = \frac{1}{5} v^{2/5}(\log v)^{1/5}\left(3-\frac{1}{(\log v)} \right) > 0\]
		
		once $\log v>1/3$, equivalently for $u> e^{e^{1/3}}>4$. Consequently $E_c$ is decreasing for $u \geq 5$.
		
		We first assume that $N\geq 5$. For each integer $j\geq 0$, let
		
		\begin{align*}
			\mathcal B_j
			&:=\left\{m\in\Nat:2^jN\leq\frac km<2^{j+1}N\right\}\\
			&=\left\{m\in\Nat:\frac{k}{2^{j+1}N}<m
			\leq\frac{k}{2^jN}\right\}\\
			&\subset\left(\frac{k}{2^{j+1}N},
			\frac{2k}{2^{j+1}N}\right].
		\end{align*}
		
		Every integer $m\leq k/N$ belongs to exactly one such block and  only finitely many blocks are nonempty, since $k/m \leq k$. Thus
		
		\[ \sum_{m\leq k/N}\frac{1}{m} E_{c}\left( \frac{k}{m}\right) = \sum_{j\geq 0} \sum_{m\in \mathcal{B}_j}\frac{1}{m} E_{c}\left( \frac{k}{m}\right) \]
		
		For $m \in \mathcal{B}_j$, $\frac{k}{m} \geq 2^j N$, then $E_{c}\left( \frac{k}{m}\right) \leq E_{c}(2^j N)$. Therefore,
		
		\[ \sum_{m\in \mathcal{B}_j}\frac{1}{m} E_{c}\left( \frac{k}{m}\right) \leq \sum_{m\in \mathcal{B}_j}\frac{1}{m} E_{c}(2^j N) \]
		
		If $0<a<1$, then 
		
		\[\sum_{a<m\leq 2a} \frac{1}{m} = 1 \leq 1 +\log 2 \]

		If $a\geq 1$, then
		
		\[\sum_{a<m\leq 2a} \frac{1}{m} \leq \frac{1}{[a]+1} + \int_{a}^{2a} \frac{dt}{t} \leq 1 +\log 2 \]
		
		Hence, uniformly in $j,k,N$,
		
		\[\sum_{m \in \mathcal{B}_j} \frac{1}{m} \leq 1 + \log 2\]
		
		It follows that
		
		\[ \sum_{m\leq k/N}\frac{1}{m} E_{c}\left( \frac{k}{m}\right) \leq (1+\log 2)\sum_{j\geq 0}  E_{c}(2^j N) \]
		
		It remains to estimate the dyadic sum. For every $j\geq 1$ and every $u \in[2^{j-1}N, 2^j N]$, we have $E_c(u) \geq E_c(2^j N)$. Therefore
		
		\[ \int_{2^{j-1}N}^{2^j N} \frac{E_c(u)}{u} du \geq E_c(2^j N) \int_{2^{j-1}N}^{2^j N} \frac{du}{u} = (\log 2) E_c(2^j N). \]
		
		Hence
		
		\[ \sum_{j\geq 1} E_c(2^j N) \leq \frac{1}{\log 2} \int_{N}^{+\infty} \frac{E_c(u)}{u} du, \]
		
		\noindent and 
		
		\[ \sum_{j\geq 0} E_c(2^j N) \leq E_c(N) + \frac{1}{\log 2} \int_{N}^{+\infty} \frac{E_c(u)}{u} du \]
		
		For $c'<c$ we have $E_c(N) \leq E_{c'}(N)$, and by Lemma \ref{Eint}
		
		\[ \int_{N}^{+\infty} \frac{E_c(u)}{u} du \ll  E_{c'}(N)\]

		Combining everything together
		
		\[ \sum_{m\leq k/N}\frac{1}{m} E_{c}\left( \frac{k}{m}\right) \ll  E_{c'}(N) \]
		
		This proves the result for $N\geq 5$. 
	\end{proof}
	
	\begin{lemma}
		\[
		\left|\frac1k\sum_{m=1}^k A_x\left(\frac mk\right)
		1_{[1/N,1]}(m/k)
		-\int_0^1A_x(t)1_{[1/N,1]}(t)\,dt\right|
		\ll\frac Nk.
		\]
	\end{lemma}
	
	\begin{proof}
		For a function $f : [a,b] \longrightarrow \Real$ define 
		
		$$\mathrm{Var}_{[a,b]}(f) := \sup_{a=t_0< t_1 < ... < t_m=b}\sum_{i=1}^{m}|f(t_i)-f(t_{i-1})|,$$
		
		\noindent where the supremum is taken over all the partitions of the interval $[a,b]$. The comparison of Riemann sums with the integral has the following well known bound:
		\[
		\left|\frac1k\sum_{m=1}^k A_x\left(\frac mk\right)
		1_{[1/N,1]}(m/k)
		-\int_0^1A_x(t)1_{[1/N,1]}(t)\,dt\right|
		\ll\frac{\mathrm{Var}_{[0,1]}(A_x1_{[1/N,1]})}{k}.
		\]
		To conclude, observe that
		\[
		\mathrm{Var}_{[0,1]}(A_x1_{[1/N,1]})
		=\mathrm{Var}_{[1/N,1]}(A_x)
		=\sum_{n=1}^{N}|\mu(n)x^n|\leq N.
		\]
		Indeed, the discontinuities of $A_x$ occur at $t=1/n$ for
		$1\leq n\leq N$.
	\end{proof}
	
	\begin{prop}
		There is a constant $c_3>0$, such that
		\[\|Q_k - g\| \ll E_{c_3}(k)\]	
	\end{prop}
	
	\begin{proof}
		Using Lemma \ref{Eint}, we have 
		
		\[ \left| \int_{0}^{1/N} A_x(t)dt \right| \ll \int_{0}^{1/N} \frac{1}{t}E_{c_1}(1/t)dt = \int_{N}^{+\infty}\frac{E_{c_1}(u)}{u}du \ll E_{c_2}(N) \]
		
		We also have
		
		\[ \left| \frac{1}{k} \sum_{m<k/N} A_x\left( \frac{m}{k} \right) \right| \ll \frac{1}{k} \sum_{m<k/N} \frac{k}{m} E_{c_1}(k/m) = \sum_{m<k/N} \frac{1}{m} E_{c_1}(k/m) \ll E_{c_2}(N) \]
		
		Combing these two bounds with the previous lemma we get:
		
		\[ ||Q_k-g|| \ll \frac{N}{k} + E_{c_2}(N) \]
		
		Let now $N=[\sqrt(k)]$, then $||Q_k-g|| \ll  E_{c_3}(k)$.
		
	\end{proof}
	
	\begin{prop}
		The series $\sum_{k\geq 1} \frac{||h_k||}{k(k+1)}$ converges, in particular the function $\sum_{k\geq 1} \frac{h_k}{k(k+1)}$ is continuous.
	\end{prop}
	
	\begin{proof}
		Using the previously established bound and the comparison of series with integral, it is enough to show that the integral
		
		\[ \int_{K}^{+\infty} \frac{1}{t} E_{c_3}(t) = \int_{\log K}^{+\infty}  \exp\left( -c_3\frac{t^{3/5}}{(\log t)^{1/5}}\right)  dt \]
		
		\noindent is convergent. However, for $t$ sufficiently large $\frac{t^{3/5}}{(\log t)^{1/5}} \geq \sqrt(t)$, and 
		
		$$\int_{1}^{+\infty} \exp(-c_3\sqrt(t))   dt < +\infty$$
	\end{proof}
	
	\section{Criterion}
	
	The main result of this section is:
	
	\begin{thm}
		\( \left\| f(u) \right\|^2 =O_\varepsilon(u^{-\varepsilon})\) for every \(\varepsilon>0\) if and only if the Riemann hypothesis holds.
	\end{thm}
	
	\begin{proof}
	Let \(D_u(s)=\sum_{n\ge1}\mu(n)e^{-un}n^{-s}\). Using the formula from section \ref{prod}, we get:
	\[\left\langle t^{1-s} \mid f(u))\right\rangle 
	=\frac{\zeta(s)}s\bigl(D_u(s)-D_u(1)\bigr).\]
	
	Assume RH. Littlewood’s reciprocal estimate gives, for every \(a,\eta>0\),
	\[
	\frac1{|\zeta(\sigma+iv)|}
	\ll_{a,\eta}(1+|v|)^\eta
	\qquad \left(\sigma\ge \tfrac12+a\right).
	\]

	Take \(s=\tfrac12+i\tau\) or \(s=1\). Mellin inversion of \(e^{-un}\), followed by the absolutely convergent Dirichlet series for \(1/\zeta\), gives
	\[
	D_u(s)=\frac1{2\pi i}\int_{\Re w=1}
	\frac{\Gamma(w)u^{-w}}{\zeta(s+w)}\,dw.
	\]
	For any \(0<a<1\), shift the contour to \(\Re w=a\). RH makes \(1/\zeta(s+w)\) holomorphic throughout the strip of the shift, and \(\Gamma(w)\) is holomorphic there. Stirling’s estimate gives exponential decay on the horizontal sides of the contour, while Littlewood’s reciprocal estimate gives at most polynomial growth for the reciprocal zeta factor. On the new line this yields
	\[
	\begin{aligned}
		|D_u(\tfrac12+i\tau)|
		&\ll_{a,\eta}
		u^{-a}\int_{\mathbb R}|\Gamma(a+iv)|(1+|\tau+v|)^\eta\,dv\\
		&\ll_{a,\eta}u^{-a}(1+|\tau|)^\eta.
	\end{aligned}
	\]
	The same argument gives \(|D_u(1)|\ll_{a,\eta}u^{-a}\).
	
	Now, apply Fourier Plancherel to \(y\mapsto e^{-y/2}f(u)(e^y)\), extended by zero for \(y<0\). With \(s=\tfrac12+i\tau\), then
	\[
	\left\| f(u) \right\|^2
	=\frac1{2\pi}\int_{\mathbb R}
	\left|\frac{\zeta(s)}s\right|^2
	\left|D_u(s)-D_u(1)\right|^2\,d\tau
	\]
	The classical convexity estimate gives \(|\zeta(\tfrac12+i\tau)|\ll_\eta(1+|\tau|)^{1/4+\eta}\). It follows,
	\[
	\left\| f(u) \right\|^2 \ll_{a,\eta}u^{-2a}(1+|\tau|)^{-3/2+4\eta}.
	\]
	Choose \(0<\eta<1/8\) to make the integral converge. This proves \(\left\| f(u) \right\|^2\ll_a u^{-2a}\) for every \(0<a<1\). Given \(\varepsilon>0\), choose \(a\) with \(2a\le\varepsilon\).
	
	Now, assume \(\left\| f(u) \right\|^2=O_\varepsilon(u^{-\varepsilon})\) for every \(\varepsilon>0\). Fix a real \(\sigma\in(\tfrac12,1)\) with \(\zeta(\sigma)\ne0\). By Cauchy–Schwarz,
	\[
	\left|\left\langle t^{1-s} \mid f(u))\right\rangle\right|
	\le
	\left\| f(u) \right\|
	\left(\int_1^\infty t^{-2\sigma}\,dt\right)^{1/2}
	=
	\frac{\left\| f(u) \right\|}{\sqrt{2\sigma-1}}.
	\]
	It follows from scalar product formula that
	\[
	D_u(\sigma)-D_u(1)=O_{\sigma,\delta}(u^{-\delta})
	\qquad\text{for every }\delta>0
	\]
	At infinity, the defining series gives \(D_u(\sigma)-D_u(1)=O_\sigma(e^{-u})\). Consequently
	\[
	G_\sigma(w)=\int_0^\infty
	\bigl(D_u(\sigma)-D_u(1)\bigr)u^{w-1}\,du
	\]
	is holomorphic for \(\Re w>0\). Indeed, on any compact subset of this half-plane, choose \(\delta\) smaller than the minimum real part on the compact set.
	
	For \(\Re w>1-\sigma\), absolute convergence permits integration term by term. Using \(\int_0^\infty e^{-un}u^{w-1}\,du=\Gamma(w)n^{-w}\), we obtain
	\[G_\sigma(w)=\Gamma(w)
	\left(\frac1{\zeta(\sigma+w)}-\frac1{\zeta(1+w)}\right).\]
	Since \(\Gamma\) is holomorphic and zero-free on \(\Re w>0\), This formula shows that
	\[
	\frac{G_\sigma(w)}{\Gamma(w)}+\frac1{\zeta(1+w)}
	\]
	is a holomorphic continuation of \(1/\zeta(\sigma+w)\) to that entire half-plane.
	
	If \(\zeta(\rho)=0\) with \(\Re\rho>\tfrac12\), choose \(\sigma\) with \(\tfrac12<\sigma<\min(1,\Re\rho)\) and \(\zeta(\sigma)\ne0\). Such a choice exists because the zeros of \(\zeta\) are isolated. At \(w=\rho-\sigma\), the meromorphic function \(1/\zeta(\sigma+w)\) has a pole, contradicting its holomorphic continuation there. Thus \(\zeta\) has no zero to the right of the critical line. The functional equation excludes nontrivial zeros to its left as well, proving the Riemann hypothesis under this assumption.
	\end{proof}

	
	\medskip
	
	\bibliographystyle{alpha}
	\addcontentsline{toc}{section}{References}
	\bibliography{Riemann_F3}

@article{Baez2003,
	author = {Báez-Duarte, Luis},
	journal = {Atti della Accademia Nazionale dei Lincei.},
	language = {eng},
	month = {3},
	number = {1},
	pages = {5-11},
	publisher = {Accademia Nazionale dei Lincei},
	title = {A strengthening of the Nyman-Beurling criterion for the Riemann hypothesis},
	url = {http://eudml.org/doc/252348},
	volume = {14},
	year = {2003},
}

@article{BAEZDUARTE19991,
	title = {A Class of Invariant Unitary Operators},
	journal = {Advances in Mathematics},
	volume = {144},
	number = {1},
	pages = {1-12},
	year = {1999},
	issn = {0001-8708},
	doi = {https://doi.org/10.1006/aima.1998.1801},
	url = {https://www.sciencedirect.com/science/article/pii/S0001870898918011},
	author = {Luis Báez-Duarte},
}

@article{BD2002,
	title = {Arithmetical versions of Nyman-Beurling Criterion for Riemann Hypothesis},
	journal = {IJMMS},
	volume = {1324},
	year = {2002},
	url = {https://arxiv.org/pdf/math/0011254},
	author = {Luis Báez-Duarte},
}

@article{BD22002,
	title = {New Versions of the Nyman-Beurling Criterion for the Riemann Hypothesis,},
	journal = {IJMMS},
	volume = {31:7},
	pages = {387-406},
	year = {2002},
	author = {Luis Báez-Duarte},
}

@book {Ten95,
	AUTHOR = {Tenenbaum, Gérald},
	TITLE = {Introduction to Analytic and Probabilistic Number Theory},
	SERIES = {Cambridge studies in advanced mathematics},
	VOLUME = {46},
	PUBLISHER = {Cambridge University Press},
	ADDRESS = {Cambridge},
	YEAR = {1995},
	ISBN = {0-521-41261-7},
	
}

@article{FRO,
	title = {Numerical studies of the Möbius power series},
	journal = {BIT Numerical Mathematics},
	year = {1966},
	author = {Froberg, Carl-Erik},
}

@article{BettinConrey2013Period,
	author  = {Bettin, Sandro and Conrey, J. Brian},
	title   = {Period Functions and Cotangent Sums},
	journal = {Algebra \& Number Theory},
	volume  = {7},
	number  = {1},
	pages   = {215--242},
	year    = {2013},
	doi     = {10.2140/ant.2013.7.215},
	eprint  = {1111.0931},
	archivePrefix = {arXiv},
	primaryClass  = {math.NT}
}

@article{BettinConrey2013Reciprocity,
	author  = {Bettin, Sandro and Conrey, J. Brian},
	title   = {A Reciprocity Formula for a Cotangent Sum},
	journal = {International Mathematics Research Notices},
	volume  = {2013},
	number  = {24},
	pages   = {5709--5726},
	year    = {2013},
	doi     = {10.1093/imrn/rns211}
}

@article{Bettin2015,
	author  = {Bettin, Sandro},
	title   = {On the Distribution of a Cotangent Sum},
	journal = {International Mathematics Research Notices},
	volume  = {2015},
	number  = {21},
	pages   = {11419--11432},
	year    = {2015},
	doi     = {10.1093/imrn/rnv036},
	eprint  = {1411.2293},
	archivePrefix = {arXiv},
	primaryClass  = {math.NT}
}

@article{MaierRassias2016,
	author  = {Maier, Helmut and Rassias, Michael Th.},
	title   = {Generalizations of a Cotangent Sum Associated to the
		{Estermann} Zeta Function},
	journal = {Communications in Contemporary Mathematics},
	volume  = {18},
	number  = {1},
	pages   = {1550078},
	year    = {2016},
	doi     = {10.1142/S0219199715500789},
	eprint  = {1410.2145},
	archivePrefix = {arXiv},
	primaryClass  = {math.NT}
}

@article{Rassias2014,
	author  = {Rassias, Michael Th.},
	title   = {A Cotangent Sum Related to Zeros of the {Estermann} Zeta
		Function},
	journal = {Applied Mathematics and Computation},
	volume  = {240},
	pages   = {161--167},
	year    = {2014},
	doi     = {10.1016/j.amc.2014.04.086}
}

@article{Vasyunin1996,
	author  = {Vasyunin, V. I.},
	title   = {On a Biorthogonal System Associated with the {Riemann}
		Hypothesis},
	journal = {St. Petersburg Mathematical Journal},
	volume  = {7},
	number  = {3},
	pages   = {405--419},
	year    = {1996},
	note    = {English translation of Algebra i Analiz 7 (1995), no. 3,
		118--135}
}

@book{Olver1997,
	author    = {Olver, Frank W. J.},
	title     = {Asymptotics and Special Functions},
	publisher = {A K Peters},
	address   = {Wellesley, MA},
	year      = {1997},
	note      = {Reprint of the 1974 Academic Press edition},
	isbn      = {978-1-56881-069-0}
}

@book{NIST2010,
	editor    = {Olver, Frank W. J. and Lozier, Daniel W. and Boisvert,
		Ronald F. and Clark, Charles W.},
	title     = {{NIST} Handbook of Mathematical Functions},
	publisher = {Cambridge University Press},
	address   = {Cambridge},
	year      = {2010},
	isbn      = {978-0-521-19225-5},
	url       = {https://dlmf.nist.gov/}
}

@book{Davenport2000,
	author    = {Davenport, Harold},
	title     = {Multiplicative Number Theory},
	edition   = {3},
	series    = {Graduate Texts in Mathematics},
	volume    = {74},
	publisher = {Springer-Verlag},
	address   = {New York},
	year      = {2000},
	note      = {Revised and with a preface by Hugh L. Montgomery},
	isbn      = {978-0-387-95097-6}
}

@article{DarsesHillion2021,
	author  = {Darses, S\'ebastien and Hillion, Erwan},
	title   = {An Exponentially Averaged {Vasyunin} Formula},
	journal = {Proceedings of the American Mathematical Society},
	volume  = {149},
	number  = {7},
	pages   = {2969--2982},
	year    = {2021},
	doi     = {10.1090/proc/15422},
	eprint  = {2004.10086},
	archivePrefix = {arXiv},
	primaryClass  = {math.NT}
}

@article{MaierRassias2018Sums,
	author  = {Maier, Helmut and Rassias, Michael Th.},
	title   = {Estimates of Sums Related to the {Nyman--Beurling}
		Criterion for the {Riemann Hypothesis}},
	journal = {Journal of Number Theory},
	volume  = {188},
	pages   = {96--120},
	year    = {2018},
	doi     = {10.1016/j.jnt.2017.12.012},
	eprint  = {1705.09921},
	archivePrefix = {arXiv},
	primaryClass  = {math.CA}
}

@article{MaierRassias2019Explicit,
	author  = {Maier, Helmut and Rassias, Michael Th.},
	title   = {Explicit Estimates of Sums Related to the
		{Nyman--Beurling} Criterion for the {Riemann Hypothesis}},
	journal = {Journal of Functional Analysis},
	volume  = {276},
	number  = {12},
	pages   = {3832--3857},
	year    = {2019},
	doi     = {10.1016/j.jfa.2018.06.022},
	eprint  = {1806.05070},
	archivePrefix = {arXiv},
	primaryClass  = {math.CA}
}

@article{Korobov1958,
	author  = {Korobov, N. M.},
	title   = {Estimates of Trigonometric Sums and Their Applications},
	journal = {Uspekhi Matematicheskikh Nauk},
	volume  = {13},
	number  = {4(82)},
	pages   = {185--192},
	year    = {1958},
	note    = {In Russian}
}

@article{Vinogradov1958,
	author  = {Vinogradov, I. M.},
	title   = {A New Estimate for {$\zeta(1+it)$}},
	journal = {Izvestiya Akademii Nauk SSSR. Seriya Matematicheskaya},
	volume  = {22},
	pages   = {161--164},
	year    = {1958},
	note    = {In Russian}
}
	\nocite{*}
	
	\textit{E-mail address}: a413xpyv@hotmail.com
\end{document}